\documentclass[11pt]{amsart}

\usepackage{latexsym,amsfonts,amssymb,verbatim,graphicx}
\usepackage{amsmath,amsthm,amssymb,latexsym,textcomp}
\usepackage{eucal,eufrak}
\usepackage{pinlabel,url}
\input{xy}
\xyoption{all}
\input{epsf.tex}

\newcommand{\nc}{\newcommand}
\nc{\nt}{\newtheorem}
\nc{\dmo}{\DeclareMathOperator}

\theoremstyle{plain}
\newtheorem{theorem}{Theorem}[section]
\newtheorem*{maintheorem}{Main Theorem}
\newtheorem{lemma}[theorem]{Lemma}
\newtheorem{corollary}[theorem]{Corollary}

\newtheorem{proposition}[theorem]{Proposition}

\nc{\R}{\mathbb R}
\nc{\C}{\mathcal C}
\nc{\Z}{\mathbb Z}
\nc{\N}{\mathbb N}
\nc{\Q}{\mathbb Q}

\nc{\h}{\mathfrak h}

\nc{\F}{\mathbb F}
\nc{\Fam}{\mathcal F}

\dmo{\ext}{ext}
\renewcommand{\int}{\mathrm{int}}
\dmo{\Area}{Area}
\nc{\G}{\mathcal{G}}
\def\PSL{\mathrm{PSL}}

\dmo{\Mod}{Mod}
\nc{\M}{\mathcal{M}}
\nc{\inj}{\mathrm{inj}}
\dmo{\vol}{Vol}

\newcommand{\entropy}{6\log(2)} 
\nc{\twoentropy}{12\log(2)}

\dmo{\Cone}{Cone}
\dmo{\Hull}{Hull}

\nc{\q}{\bf q}

\nc{\p}[1]{\medskip\paragraph{{\em #1}}}
\nc{\margin}[1]{\marginpar{\scriptsize #1}}

\title[Stretch factors and homology]{Pseudo-Anosov stretch factors and homology of mapping tori}

\begin{document}
	
\input{epsf.sty}

\author{Ian Agol}

\author{Christopher J. Leininger}

\author{Dan Margalit}

\address{Ian Agol \\ Department of Mathematics \\ University of California Berkeley \\ Berkeley, CA}

\address{Christopher J. Leininger\\ Dept. of Mathematics, University of Illinois at Urbana--Champaign \\ 273 Altgeld Hall, 1409 W. Green St. \\ Urbana, IL 61802\\ clein@math.uiuc.edu}

\address{Dan Margalit \\ School of Mathematics\\ Georgia Institute of Technology \\ 686 Cherry St. \\ Atlanta, GA 30332 \\  margalit@math.gatech.edu}

\thanks{This material is based upon work supported by the National Science Foundation under Grant Nos. DMS - 1406301, DMS - 1207183, and DMS - 1057874.}



\begin{abstract}
We consider the pseudo-Anosov elements of the mapping class group of a surface of genus $g$ that fix a rank $k$ subgroup of the first homology of the surface.  We show that the smallest entropy among these is comparable to $(k+1)/g$.  This interpolates between results of Penner and of Farb and the second and third authors, who treated the cases of $k=0$ and $k=2g$, respectively, and answers a question of Ellenberg.  We also show that the number of conjugacy classes of pseudo-Anosov mapping classes as above grows (as a function of $g$) like a polynomial of degree $k$.  
\end{abstract}

\maketitle


\section{Introduction}

Let $S_g$ denote a closed, orientable surface of genus $g \geq 2$ and $\Mod(S_g)$ its mapping class group.  
The goal of this paper is to compare two numbers associated to a pseudo-Anosov element $f$ of $\Mod(S_g)$:
\begin{itemize}
 \item $\kappa(f)$, the dimension of the subspace of $H_1(S_g;\R)$ fixed by $f$, and
 \item $h(f)$, the entropy of $f$.
\end{itemize}
Note that $0 \leq \kappa(f) \leq 2g$ and that $h(f)$ equals the logarithm of the stretch factor (or dilatation) $\lambda(f)$ (see \cite{primer} for the basic definitions).  Also, $\kappa(f)+1$ is the first betti number of the mapping torus associated to $f$.

For every $(k,g)$ with $g \geq 2$ and $0 \leq k \leq 2g$, we set
\[ L(k,g) = \textrm{min} \{ h(f) \mid f \colon S_g \to S_g \mbox{ and } \kappa(f) \geq k \}.\]

For two real-valued functions $F$ and $G$, we write $F \asymp G$ if there is a universal constant $C$ so that $F/C \leq G \leq CF$.  

\begin{maintheorem}  
The function $L(k,g)$ satisfies $L(k,g) \asymp \frac{k+1}{g}$.
\end{maintheorem}

Explicit constants are given in Theorem~\ref{T:homology_bound} below. The seminal result about $L(k,g)$ is due to Penner \cite{Pe}, who proved that $L(0,g) \asymp 1/g$.  With Farb, the second and third authors proved $L(2g,g) \asymp 1$.  Ellenberg asked \cite{Ellenberg} whether the $L(k,g)$ interpolate between $L(0,g)$ and $L(2g,g)$ in the sense that $L(k,g) \geq C(k+1)/g$ for some constant $C$.  Our result in particular answers his question while also bounding $L(k,g)$ from above.

\p{Explicit constants} For $f \in \Mod(S_g)$ and any field $\F$, we can define $\kappa_\F(f)$ to be the dimension of the subspace of $H_1(S_g;\F)$ fixed by $f$ and we can define $L_\F(k,g)$ similarly to $L(k,g)$, with $\kappa$ replaced by $\kappa_\F$.  The following theorem gives explicit constants as demanded by our main theorem and also generalizes our main theorem to arbitrary field coefficients.  

\begin{theorem}\label{T:homology_bound} Let $\F$ be any field.  For all $g \geq 2$ and $0 \leq k \leq 2g$ we have:
\[  .00031\left ( \frac{k+1}{2g-2} \right ) \leq L_\F(k,g) \leq \twoentropy \left ( \frac{k+1}{2g-2} \right ) .\]
\end{theorem}
We use the denominator $2g-2$ instead of $g$ here because in our proofs of both the upper and lower bounds, it will be natural to consider the normalized entropy of a pseudo-Anosov mapping class $f \in \Mod(S_g)$, namely, $|\chi(S_g)| h(f)$.

\p{Comparison with previously known constants}   The best known constants for the theorems about $L(0,g)$ and $L(2g,g)$ are as follows:
\begin{alignat*}{6}
\frac{\log(2)}{6} \left ( \frac{1}{2g-2} \right ) & \ \ \leq \ \  & L(0,g) \, \,  & \ \ \leq \ \ && \log \left( \varphi^4 \right) \left ( \frac{1}{2g-2} \right ) \\
.197 & \ \ \leq \ \ &  L(2g,g) & \ \ \leq \ \ && \log(62)
\end{alignat*} 
where $\varphi = (1+\sqrt{5})/2$ is the golden ratio.
The constant $\log(2)/3 \approx .231$ is due to Penner \cite{Pe}; see also McMullen \cite{Mc}.  The constant $\log(\varphi^4) \approx 1.925$ comes from the work of Aaber--Dunfield \cite{AD}, Hironaka \cite{Hir}, and Kin--Takasawa \cite{KT} who independently constructed examples proving
\[ \limsup_{g \to \infty} \ g \cdot L(0,g) \leq \log \left(  \varphi^2 \right).\]
This statement does not immediately imply the upper bound above, but appealing to Hironaka's construction and the work of Thurston \cite{ThNorm}, Fried \cite{Fr}, and McMullen \cite{Mc} we can use elementary calculus to promote this asymptotic statement to the given bound for all $g$; see Proposition~\ref{P:limsup to sup} in the appendix.  The constants $.197$ and $\log(62) \approx 4.127$ for $L(2g,g)$ are due to Farb and the second and third authors \cite{FLM1}.

A priori, we have $L(0,g) \leq L(k,g) \leq L(2g,g)$ and so the previously known bounds automatically give:
\[ \frac{\log(2)}{6} \left ( \frac{1}{2g-2} \right ) \leq  L(k,g) \leq \log(62).\]
The upper bound in Theorem~\ref{T:homology_bound} improves on the upper bound of $\log(62)$ as soon as $g$ is slightly larger than $k$ (specifically, $g > (k+1) \log(64)/\log(62) +1$).  In order for our Theorem~\ref{T:homology_bound} to improve on Penner's lower bound, we need $k > 372$, hence $g > 186$.   On the other hand, the argument for the lower bound can be applied under fairly mild hypotheses to provide an improvement on Penner's lower bound; see Corollary~\ref{C:CSmod2}.  Of course, the point of Theorem~\ref{T:homology_bound} is the asymptotic behavior.

\p{Fixed subspaces of a fixed dimension.} In the definition of $L(k,g)$ one might be inclined to replace the inequality with an equality, that is, to consider the smallest entropy among pseudo-Anosov elements of $\Mod(S_g)$ fixing a subspace of $H_1(S_g;\R)$ whose dimension is exactly equal to $k$.  Our Main Theorem is still valid with this definition; we explain the necessary modifications to the proof at the end of Section~\ref{sec:count}.  In Penner's original work there are no constraints on the action of homology, and in the work of the last two authors with Farb we have the strictest possible constraint.  As such, our  $L(k,g)$ best interpolates between these two situations.

\p{Congruence subgroups} Let $\Mod(S_g)[m]$ denote the level $m$ congruence subgroup of $\Mod(S_g)$, that is, the kernel of the natural action of $\Mod(S_g)$ on $H_1(S_g;\Z/m\Z)$.  We write $L(\Mod(S_g)[m])$ for the minimum entropy over all pseudo-Anosov elements of $\Mod(S_g)[m]$.   In work with Farb, the second and third author showed for fixed $m \geq 3$ that $L(\Mod(S_g)[m]) \asymp 1$; \cite{FLM1}.  Setting $k=2g$ and $\F=\F_2$ (the field with $2$ elements), $L_{\F_2}(2g,g) = L(\Mod(S_g)[2])$, and so our main theorem extends this result to the case of $m=2$, cf.~\cite[Question 2.9]{FLM1}.

\medskip

\p{Counting conjugacy classes}  The following theorem can be viewed as a refinement of a theorem of the second and third authors \cite[Theorem 1.3]{LM}.  In the statement, let $\mathcal{G}_{g,k}(L)$ denote the number of conjugacy classes of pseudo-Anosov mapping classes $f \in \Mod(S_g)$ with $\kappa(f) = k$ and $h(f) < L \cdot (k+1)/(2g-2)$.

\begin{theorem} \label{T:homology_count}
Let $k \geq 0$ and let $L \geq \entropy$.  There are constants $c_1,c_2>0$ so that
\[ |\mathcal{G}_{g,k}(L)| \geq c_1g^k-c_2. \]
\end{theorem}

Conjugacy classes of pseudo-Anosov homeomorphisms in $\Mod(S_g)$ are in bijection with geodesics in the moduli space of Riemann surfaces of genus $g$ and the entropy of a pseudo-Anosov mapping class is equal to the length of the corresponding geodesic in the Teichm\"uller metric on moduli space.  There is a canonical surface bundle on moduli space whose monodromy is given by the correspondence between loops in moduli space and elements of $\Mod(S_g)$.    Theorem~\ref{T:homology_count} can thus be viewed as an estimate on the number of short geodesics in moduli space where the associated monodromy fixes a subspace of dimension $k$.  

\p{Acknowledgments} We would like to thank Jeffrey Brock, Ken Bromberg, Sadayoshi Kojima, and Greg McShane for helpful conversations.
 
\section{Lower bounds}

The following proposition is the main goal of this section.  It implies the lower bound in Theorem~\ref{T:homology_bound} and in fact generalizes it to the case of an arbitrary surface of finite type (that is, a surface obtained from a compact surface by deleting finitely many points from the interior).

\begin{proposition} \label{P:homology_bound_gen}
Let $S$ be a surface of finite type and $\F$ a field.  If $f \in \Mod(S)$ is pseudo-Anosov, then 
\[ .00031 \left(\frac{\kappa_\F(f) + 1}{|\chi(S)|}\right ) \leq h(f).\]
\end{proposition}

The proof requires a few preliminary facts.  For a hyperbolic 3-manifold $M$ with finite volume, we denote by $\vol(M)$ the hyperbolic volume.  Also, for any field $\F$ we denote by $b_1(M;\F)$ the first Betti number of $M$ with coefficients in $\F$.

\begin{proposition} \label{P:volume_b1}
Let $\F$ be any field.  If $M$ is a complete, orientable, hyperbolic $3$-manifold of finite volume, then 
\[ b_1(M;\F) \leq 334.08 \cdot \vol(M). \] 
\end{proposition}

Gelander proved  \cite[Corollary 1.3]{Gelander} there exists a constant $D > 0$ so that $b_1(M;\F) \leq D \cdot \vol(M)$ (see also \cite{BGLS}).  Gelander's result applies in much greater generality, but the proof there does not provide an explicit constant. 

For a Riemannian manifold, the $\epsilon$-thick part $M_{\geq \epsilon}$ is the subset of $M$ with injectivity radius greater than or equal to $\epsilon/2$.  An $\epsilon > 0$ is a \emph{Margulis constant} for a hyperbolic manifold $M$ if
\begin{itemize}
\item $M_{\geq \epsilon}$ is the complement of a union of open solid tori and product neighborhoods of the ends of $M$, and
\item for every $x \in M_{\geq \epsilon}$, the open ball $B_{\epsilon/2}(x)$ in $M$ is path isometric to an open ball of radius $\epsilon/2$ in hyperbolic $3$--space.
\end{itemize}

To prove Proposition~\ref{P:volume_b1}, we will need a theorem of the first author with Culler and Shalen \cite[Corollary 4.2 and Lemma 5.2]{ACS}.

\begin{theorem}[Agol--Culler--Shalen]
\label{thm:acs}
Let $M$ be a complete hyperbolic manifold of finite volume.
\begin{enumerate}
\item If each 2-generator subgroup of $\pi_1(M)$ is either free or free abelian, then $\log(3)$ is a Margulis number for $M$.
\item If $M$ satisfies either of the conditions
\begin{itemize}
 \item $b_1(M;\Q) \geq 3$, or
 \item $M$ is closed and $b_1(M;\F_p) \geq 4$ for some prime $p$,
\end{itemize}
then each 2-generator subgroup of $\pi_1(M)$ is either free or free abelian; in particular, $\log(3)$ is Margulis constant for $M$.
\end{enumerate}
\end{theorem}

We will also need to extend this theorem to one additional case.

\begin{proposition}
\label{P:ACSlog3}
Let $p$ be any prime.  If $M$ is a complete, noncompact hyperbolic manifold of finite volume with $b_1(M;\F_p) \geq 5$, then $log(3)$ is a Margulis number for $M$.
\end{proposition}

\begin{proof}

By the structure of the thick-thin decomposition of $M$ \cite[Corollary 5.2.10]{Th}, $M$ is homeomorphic to the interior of a compact manifold $M'$ with boundary $\partial M'$ homeomorphic to a union of $d \geq 1$ tori.  From Poincar\'e--Lefschetz duality and the long-exact sequence of the homology of the pair $(M',\partial M')$ it follows that the image of $H_1(\partial M';\F) \to H_1(M';\F) \cong H_1(M;\F)$ has dimension exactly $d$ for any field $\F$.   Thus, if $d \geq 3$ then $b_1(M;\Q) \geq 3$, and we may appeal to Theorem~\ref{thm:acs}(2).

It remains to consider the cases of $d \in \{1,2\}$.  Let $G$ be any 2-generator subgroup of $\pi_1(M)$.  The image of $H_1(G;\F_p)$ in $H_1(M;\F_p)$ is a subspace of dimension at most $2$, which together with the image of $H_1(\partial M';\F_p)$ spans a subspace of dimension at most $2+d$ in $H_1(M;\F_p)$.  Therefore, we can find a homomorphism from $\pi_1(M)$ onto $\F_p^r$, where $r \geq b_1(M;\F_p)-(2+d) \geq 3-d$ whose kernel contains both $G$ and $\pi_1(\partial M'_0)$ for each component $\partial M'_0 \subset \partial M'$.  Then the cover $\widetilde M' \to M'$ corresponding to the kernel has $dp^r \geq kp^{3-d} \geq 4$ boundary components and has fundamental group containing an isomorphic copy of $G$.  Appealing again to Theorem~\ref{thm:acs} completes the proof.
\end{proof}

\begin{proof}[Proof of Proposition \ref{P:volume_b1}]

Let $\F$ be any field.  To prove the proposition, we may assume that $b_1(M;\F) \geq 5$; indeed, the volume of any orientable hyperbolic 3-manifold is at least that of the Weeks manifold, which is approximately $.94$ \cite{GMM} and so $334.08 \cdot \vol(M)$ is always (much) greater than five.  When $b_1(M;\F) \geq 5$, then $b_1(M;\F_p) \geq 5$ for some prime $p$ (indeed, a matrix representing the action of $f$ on  $H_1(S_g;\F)$ has coefficients in the prime subfield of $\F$, and thus $\kappa_\F(f)$ depends only on the characteristic of $\F$).  Thus by Proposition \ref{P:ACSlog3}, $\epsilon = \log(3)$ is a Margulis number for $M$.

Since $\pi_1(M_{\geq \epsilon}) \to \pi_1(M)$ is surjective it suffices to show that $b_1(M_{\geq \epsilon};\F) \leq 334.08 \cdot \vol(M)$.  To this end, let $V = \{v_1, \dots, v_m \}$ be a maximal collection of points in $M_{\geq \epsilon}$ so that the distance in $M$ between any two points of $V$ is at least $\epsilon$ (that is, an $\epsilon$-net).   By maximality, $M_{\geq \epsilon}$ is contained in the $\epsilon$-neighborhood of $V$ in $M$.  Let $X$ be the Voronoi cell decomposition of $M$ defined by $V$ and $\Gamma$ the dual graph (i.e. the Delaunay graph).  Let $\Gamma_0 \subset \Gamma$ be the subgraph consisting of all edges of length at most $2 \epsilon$.  To complete the proof, we will show two things:
\begin{enumerate}
 \item the map $\pi_1(\Gamma_0,v_0) \to \pi_1(M_{\geq \epsilon},v_0)$ is surjective, and 
 \item $b_1(\Gamma_0;\F) \leq 334.08 \cdot \vol(M)$.
\end{enumerate}
To prove the first statement, let $\gamma$ be any loop in $M_{\geq \epsilon}$ based at $v_0$.  We must find a loop in $\Gamma_0$ that is homotopic to $\gamma$ as a based loop.  Modifying $\gamma$ by homotopy in $M_{\geq \epsilon}$, we can assume $\gamma$ is transverse to the 2-skeleton of $X$; in particular, the intersection of $\gamma$ with any 2-cell lies in the interior of the 2-cell and there are finitely many such intersections.   Since $\gamma$ is contained in $M_{\geq \epsilon}$, it is always within distance $\epsilon$ of some $v_i$.  Therefore, any $2$--cell of $X$ that $\gamma$ crosses is dual to an edge with length at most $2\epsilon$.  It follows that we can homotope $\gamma$ (rel $v_0$) to the loop $\gamma' \subset \Gamma_0$ that traverses the edges determined by the $2$-cells $\gamma$ meets.

Now we bound $b_1(\Gamma_0;\F)$.  The number of vertices of $\Gamma_0$ is $m$. By definition of the $\{v_i\}$ and the fact that $\epsilon$ is a Margulis number for $M$ the $\epsilon/2$-balls $\{ B_{\epsilon/2}(v_j)\}_{j=1}^m$ are embedded and pairwise disjoint.  We thus have
\[ m \vol\left(B_{\epsilon/2}\right)  \leq \vol(M).\]
And using the formula $\vol(B_r) = \pi(\sinh 2r-2r)$ we calculate
\[ m < \vol(M)/0.234721.\]

The valence of any vertex of $\Gamma_0$ is bounded by the maximum number of points in a ball of radius $2 \epsilon$ with distance to the center at least $\epsilon$ and pairwise distance at least $\epsilon$.  This is the same as the maximum number of pairwise disjoint balls of radius $\epsilon/2$ that we can fit in the shell outside a ball of radius $\epsilon/2$ and inside a ball of radius $5 \epsilon/2$ with the same center.
This is bounded by the volume ratio
\[ \mathcal V = \frac{\vol(B_{5 \epsilon/2}) - \vol(B_{\epsilon/2})}{\vol(B_{\epsilon/2})} < 493.2244575.\] 
The valence of $\Gamma_0$ is thus at most $\mathcal V$, and so $\Gamma_0$ has at most $\frac{\mathcal V}2 m$ edges.

The first betti number of a connected graph is given by $1 - (\# \text{vertices}) + (\# \text{edges})$.   Thus by the previous two paragraphs, 
\[ b_1(\Gamma_0;\F) \leq 1+ \left(\frac{\mathcal V-2}{2}\right)m < 1 + \frac{491.2244575}{2\pi (0.234721)}\vol(M) < 1 + 333.08\vol(M). \]
Using the fact that a complete, hyperbolic manifold $M$ with $b_1(M;\F) \geq 5$ has volume at least 1 (see \cite{ACS,CM,PrzB1}), the proposition follows.
\end{proof}

\p{Remark} In the previous proof, we used a Voronoi decomposition instead of a good open cover because the latter would require an $\epsilon/2$-net instead of an $\epsilon$-net; this would result in a larger constant.

\bigskip

Proposition~\ref{P:homology_bound_gen} will follow easily from Proposition~\ref{P:volume_b1} and the following recent theorem of Kojima--McShane \cite{KM}; see also Brock-Bromberg \cite{BBWP}.   Kojima and McShane's theorem refines an earlier result of Brock \cite{BrockWP}, and builds on work of Schlenker \cite{S}.  To state it, we require a theorem of Thurston \cite{Tharxiv} which states that if $S$ is a hyperbolic surface of finite type and $f \in \Mod(S)$ is pseudo-Anosov, then the mapping torus $M = M_f$ admits a complete hyperbolic metric of finite volume; in particular, $\vol(M_f)$ makes sense and is finite.

\begin{theorem}  [Kojima--McShane] \label{T:Koj-McS}
Let $S$ be a surface of finite type.  If $f \in \Mod(S)$ is pseudo-Anosov, then 
\[ \vol(M_f) \leq 3 \pi \,|\chi(S)|\, h(f).\]
\end{theorem}

\begin{proof}[Proof of Proposition~\ref{P:homology_bound_gen}]

It is an easy application of the Mayer--Vietoris long exact sequence that $b_1(M;\F) = \kappa_\F(f) + 1$ (decompose $S^1$ into the union of two intervals and pull this decomposition back to $M_f$).  Applying this fact, Proposition~\ref{P:volume_b1}, and Theorem~\ref{T:Koj-McS} in succession,  we have
\[ \kappa_\F(f)+1 = b_1(M;\F) \leq 334.08 \cdot \vol(M) \leq 334.08\cdot 3\pi\,|\chi(S)| \, h(f).\]
Since $1/(3 \pi \cdot  334.08) \approx .000317$, the proposition follows.
\end{proof}

As suggested to us by Peter Shalen, we can combine Theorem~\ref{T:Koj-McS} with the following result of Culler--Shalen to obtain an improvement on Penner's lower bound under a mild assumption on $f$.

\begin{theorem}[Culler--Shalen]
If $M$ is a hyperbolic 3-manifold with $b_1(M;\F_2) \geq 6$ then $\vol(M) \geq 3.08$.
\end{theorem}

\begin{corollary} \label{C:CSmod2}
Let $S$ be any surface of finite type and let $f \in \Mod(S)$.  If $\kappa_{\F_2}(f) \geq 5$, then $h(f) \geq  .326/|\chi(S)|$.
\end{corollary}


\section{Upper bounds} \label{S:examples}

The next proposition gives the upper bound in Theorem~\ref{T:homology_bound}.  Combining this proposition with Proposition~\ref{P:homology_bound_gen} immediately gives Theorem~\ref{T:homology_bound}.

\begin{proposition} \label{P:homology_bound_gen 2}
Let $\F$ be a field, let $g \geq 2$, and let $0 \leq k \leq 2g$.  There exists a pseudo-Anosov $f_{g,k} \in \Mod(S_g)$ with $\kappa_\F(f_{g,k}) \geq k$ and 
\[h(f_{g,k})  \leq \twoentropy \left(\frac{k + 1}{2g-2}\right ) .\]
\end{proposition}

\p{Generating examples} We start by defining a collection of mapping classes
\[ \{f_g \in \Mod(S_g) \mid g \geq 2\} \cup \{f_2' \in \Mod(S_2)\} \]
from which all of our examples $f_{g,k}$ required for Proposition~\ref{P:homology_bound_gen 2} will be generated.  More precisely, each $f_{g,k}$ will be realized as the monodromy from some fibering of some mapping torus $M_{f_g}$ with $g = \lceil k/2 \rceil$ or of the mapping torus $M_{f_2'}$. 

For each $g \geq 3$, we will describe explicit multicurves $A_g$ and $B_g$ in $S_g$, and then $f_g$ will be the difference of the two multitwists:
\[ f_g = T_{A_g} T_{B_g}^{-1}. \]
All of the $A_g,B_g$ pairs for $g \geq 3$ will be modeled on $A_3,B_3$ in $S_3$ shown in the left-hand drawing here:

\vspace*{2ex}

\begin{center}
\labellist
\small\hair 2pt
 \pinlabel {$\alpha_\ell$} [ ] at 114 48
 \pinlabel {$\beta_\ell$} [ ] at 113 11
 \pinlabel {$\alpha_m$} [ ] at 168 57
 \pinlabel {$\beta_m$} [ ] at 168 -1
 \pinlabel {$\alpha_r$} [ ] at 227 45
 \pinlabel {$\beta_r$} [ ] at 227 11
 \pinlabel {$\alpha$} [ ] at 71 -2
 \pinlabel {$\beta$} [ ] at 25 61
 \pinlabel {$\alpha$} [ ] at 189 -2
 \pinlabel {$\beta$} [ ] at 143 61
 \pinlabel {$\alpha$} [ ] at 241 46
 \pinlabel {$\beta$} [ ] at 241 11
 \pinlabel {$\alpha$} [ ] at 266 0
 \pinlabel {$\beta$} [ ] at 266 58
 \endlabellist
\noindent \includegraphics[scale=1.2]{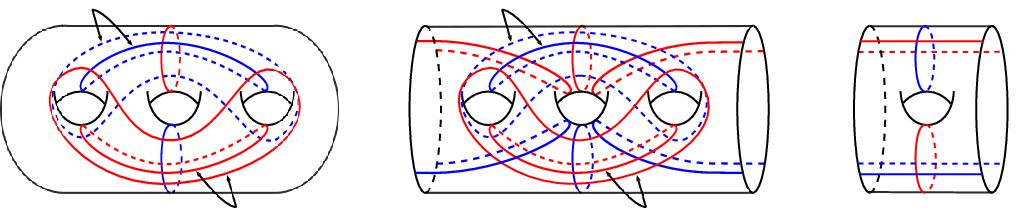}
\end{center}

\vspace*{2ex}

We will build most of the other examples for $g \geq 3$ by gluing copies of the marked surfaces shown in the middle and right-hand pictures above (by a marked surface we mean a surface with labeled curves and arcs in it).

Let $P$ denote the marked surface in the middle and $Q$ the marked surface on the right.  We will need several versions of $P$; for any subset $J$ of $\{\ell,m,r\}$ let $P(J)$ denote the marked surface obtained from $P$ by excluding all curves and arcs whose labels have subscripts not in $J$; we will write, for instance, $P(\ell,r)$ instead of $P(\{\ell,r\})$.  The two pairs of curves labeled simply by $\alpha$ and $\beta$ will appear in each copy of $P$.

For $X_1,\dots, X_n \in \{P(J) \mid J \subseteq \{\ell,m,r\} \} \cup \{Q\}$, we will write $X_1 + \cdots + X_n$ for the marked surface obtained by gluing the $X_i$ end to end and capping the boundary of the resulting surface with two disks (this operation is not commutative!).  We will also write $c \cdot X_i$ for the sum of $c$ copies of $X_i$.  Because the six arcs in $P$ and $Q$ are consistently labelled, such a sum always results in a marked surface where each curve can be consistently labeled as an $\alpha$-curve or a $\beta$-curve, as long as each $P(J)$-piece with $r \in J$ is followed by either a $Q$-piece or a $P(J)$-piece with $\ell \in J$ (and similarly for the $P(J)$ with $\ell \in J$).
For example, $P(r) + Q + P(\ell)$ is the following marked surface of genus seven:

\begin{center}
\labellist
\small\hair 2pt
\endlabellist
\noindent \includegraphics[scale=1.2]{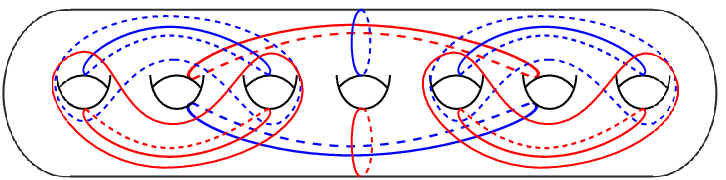}
\end{center}

With this notation in hand, we can describe all of our $A_g$ and $B_g$ with $g \geq 3$ and $g \notin \{4,5,8\}$ as follows.  For $g=3$ we take $P(m)$ as above.  Then for any $k \geq 2$, we take
\[
\begin{cases}
P(m,r) + (k-2) \cdot P(\ell,r) + P(\ell) & g = 3k \\
P(r) + Q + (k-2) \cdot P(\ell,r) + P(\ell) & g = 3k+1 \\
P(r) + Q  + (k-2) \cdot P(\ell,r) + Q+ P(\ell) & g = 3k+5
\end{cases}
\]
We define $A_g$ and $B_g$ as the unions of the $\alpha$-curves and the $\beta$-curves, respectively.  

The cases of $g \in \{4,5,8\}$ are described by the following pictures:

\vspace*{2ex}

\begin{center}
\labellist
\small\hair 2pt
\endlabellist
\noindent \includegraphics[scale=1.2]{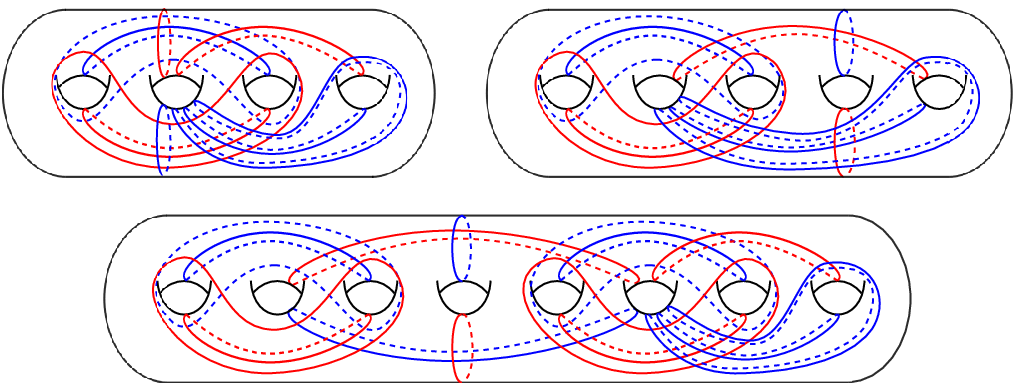}
\end{center}

\vspace*{2ex}

Here we did not label the curves as $\alpha$-curves or $\beta$-curves, but assigning the label of $\alpha$ or $\beta$ to any one of the curves determines the labels on the other curves (and it does not matter which of the two labelings we use).

Finally for $g=2$, we have $A_2,B_2$ and $A_2',B_2'$ shown in the left- and right-hand sides of the following figure, respectively:

\vspace*{4ex}

\labellist
\small\hair 2pt
 \pinlabel {$\beta$} [ ] at 50 55
 \pinlabel {$\alpha$} [ ] at 23 17
 \pinlabel {$\alpha$} [ ] at 180 55
 \pinlabel {$\beta_1$} [ ] at 142 14
 \pinlabel {$\beta_0$} [ ] at 169 34
\endlabellist
\centerline{\includegraphics[scale=1.1]{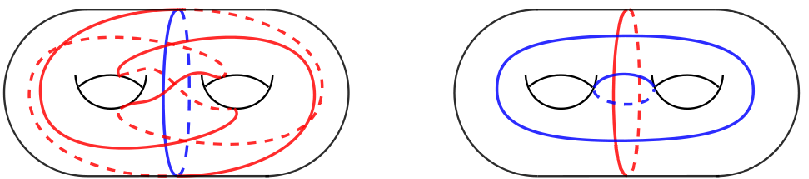}}

\vspace*{2ex}

\noindent and we define $f_2 = T_{A_2}T_{B_2}$ and $f_2' = T_{A_2'}T_{B_2'}$ (all twists are positive here).

\bigskip

\p{Properties of the generating examples} Our next goal is to bound from above the entropies of the $f_g$.  We will require the following theorem of Thurston \cite[Theorem 7]{ThMCG}.  For the statement we say that a pair of (minimally-intersecting) multicurves fills a closed surface if it cuts the surface into a disjoint union of disks.

\begin{theorem}[Thurston]
\label{theorem:thurston construction}
If $A=\{\alpha_1,\dots,\alpha_m\}$ and $B=\{\beta_1,\dots,\beta_n\}$ are multicurves that fill a surface $S$, then there is a homomorphism $\langle T_A , T_B \rangle \to \PSL(2,\R)$ given by
\[ T_A \mapsto \left( \begin{array}{cc} 1 & \sqrt{\mu} \\ 0 & 1\\ \end{array} \right) \quad \text{and}\quad T_B \mapsto \left( \begin{array}{cc} 1 & 0\\ -\sqrt{\mu} & 1\\ \end{array} \right), \]
where $\mu$ is the Perron--Frobenius eigenvalue of $NN^T$ and $N$ is the matrix $N_{ij} =  i(\alpha_i,\beta_j)$.  Moreover, an element of $\langle T_A , T_B \rangle$ is pseudo-Anosov if and only if its image is hyperbolic and in that case the entropy equals the logarithm of the spectral radius of the image.
\end{theorem}

\begin{lemma}
\label{lemma:dil}
Let $g \geq 2$.  Then $f_g$ is pseudo-Anosov with $h(f_g) < \entropy$.  Also, $f_2'$ is pseudo-Anosov with $h(f_2') < 2$.
\end{lemma}

\begin{proof}

We first deal with the case $g \geq 3$.  Under the map $\langle T_{A_g} , T_{B_g} \rangle \to \PSL(2,\R)$ from Theorem~\ref{theorem:thurston construction}, the image of $f_g$ is
\[  \left( \begin{array}{cc} 1 & \sqrt{\mu_g} \\ 0 & 1\\ \end{array} \right) \left( \begin{array}{cc} 1 & 0\\ -\sqrt{\mu_g} & 1\\ \end{array} \right)^{-1} \]
where $\mu_g$ is the Perron--Frobenius eigenvalue of the matrix $NN^T$ for $A_g$ and $B_g$.  
The trace of this matrix is $2+\mu_g$.  Thus by Theorem~\ref{theorem:thurston construction} the entropy $h(f_g)$ is bounded above by $\log(\mu_g+2)$.  Therefore, to prove the lemma it is enough to show that $\mu_g \leq 62$ for all $g\geq 3$.

For the special cases $g \in \{3,4,5,8\}$ we can use a computer to check that $\mu_g < 62$; the largest is $\mu_8 \approx 61.978$.

The remaining values of $g \geq 3$ are covered by our general construction.  There are three cases, according to the residue modulo three, and we treat each in turn.  In these cases we will bound $\mu_g$ from above using the fact that the Perron--Frobenius eigenvalue of a Perron--Frobenius matrix is bounded from above by the maximum row sum of the matrix.

To determine the row sums of $NN^T$, we draw the labeled bipartite graph associated to $A_g$ and $B_g$.  This graph has one vertex for each connected component of $A_g$, one vertex for each connected component of $B_g$, and an edge between two vertices corresponding to simple closed curves with nonzero intersection; the labels are the geometric intersection numbers.  Here is the picture for the case of $g=3k$ (the labels here are suppressed since they are all equal to two):

\begin{center}
\noindent \includegraphics[scale=.4]{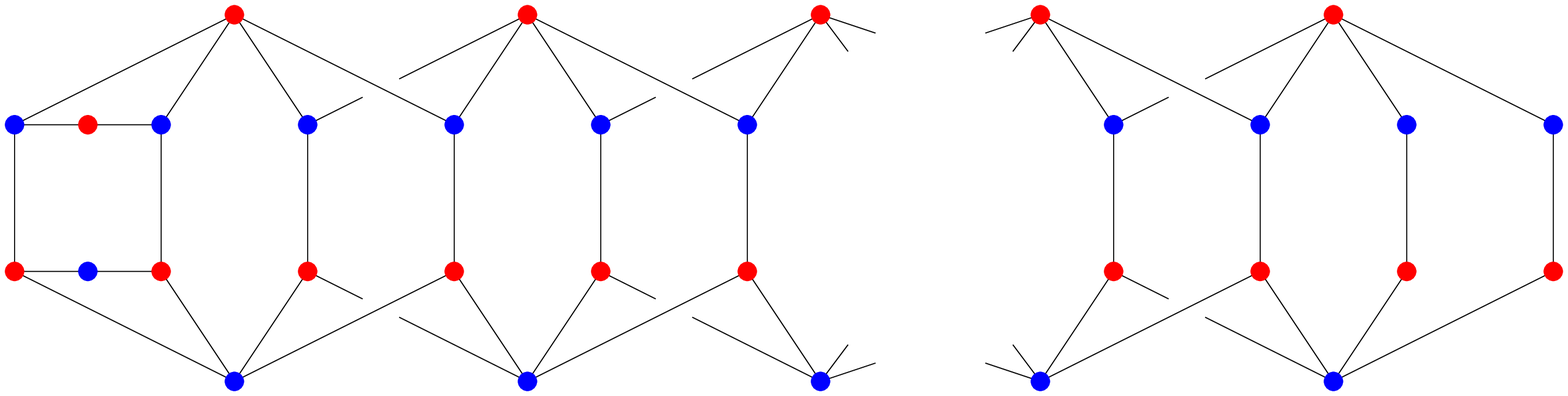}
\end{center}

In general, to compute the row sum for $NN^T$ corresponding to a particular $\alpha$-curve we take the corresponding vertex of the bipartite graph and for each (possibly-backtracking) path of length two starting from that vertex we take the product of the two labels; the sum over all paths of length two is the row sum.  In the case where all of the graphs have the same label $n$, this number is equal to $n^2$ times the sum of the degrees of the vertices adjacent to the given vertex.

In the case of the graph given above for $g = 3k$ the degree of every vertex is bounded above by 4, and moreover, every vertex of degree $4$ is only adjacent to vertices of degree at most $3$.  Since the labels are all equal to 2, each row sum of $NN^T$ is at most $2^2 \cdot 3 \cdot 4 = 48<62$, as desired.

Next we consider the case where $g = 3k+2$ and $k \geq 3$.  In this case, the graph takes the following form (again all labels are equal to 2):

\begin{center}
\noindent \includegraphics[scale=.4]{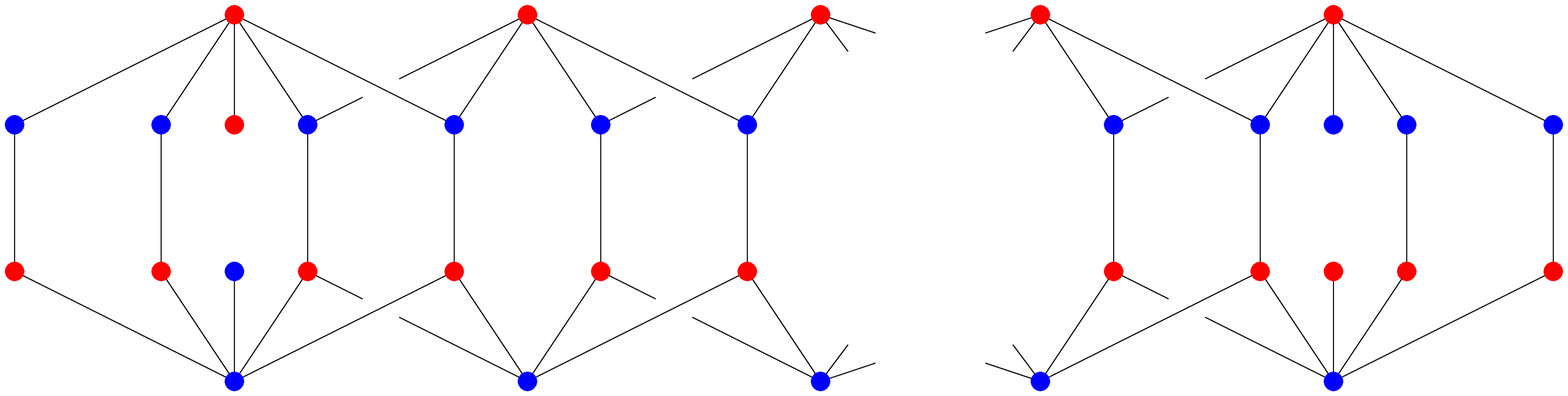}
\end{center}

In this case, there are vertices in the graph with degree 5, but the vertices adjacent to those have degrees equal to 1, 2, 2, 3, and 3, and so the row sums corresponding to the vertices of degree 5 are at most $2^2 \cdot (1+2+2+3+3)=44$.  The vertices of valence $4$ are only adjacent to vertices of valence at most $3$, so the row sums corresponding to vertices of valence 4 are again at most $48$ as in the previous case.   Finally, vertices of degree at most $3$ are adjacent to at most two vertices of degree at most 5 and a vertex of degree at most 3, and so the row sum associated to any remaining vertex is at most $2^2 \cdot (5+5+3) = 52$.  Thus all row sums are again less than 62.

For the case where $g = 3k+1$, the corresponding graph is a subgraph of the one for $g' = 3k'+2$ for any $k' \geq k$, and so the row sums are bounded above by the maximum row sums in that case, and are thus less than 62, as desired.

It remains to treat the cases of $f_2$ and $f_2'$.  In those cases the image in $\PSL(2,\R)$ is
\[  \left( \begin{array}{cc} 1 & \sqrt{\mu} \\ 0 & 1\\ \end{array} \right) \left( \begin{array}{cc} 1 & 0\\ -\sqrt{\mu} & 1\\ \end{array} \right) \]
and so the entropies are bounded above by $\log(\mu_2-2)$ and $\log(\mu_2'-2)$, where $\mu_2$ and $\mu_2'$ are the associated Perron--Frobenius eigenvalues.  For $f_2$ we have $NN^T = (8)(8)=(64)$ and so $\mu_2 = 64$; hence $h(f_2) < \log(62)$.  Similarly for $f_2'$ we have $NN^T = (2 \,\, 2)(2 \,\, 2)^T = (8)$, implying $\mu_2' = 8$ and hence $h(f_2') < \log(6) < 2$.  
\end{proof}

For the next lemma, recall that for any $\vec v \in H_1(S_g)$ there is a transvection $\tau_{\vec v} : H_1(S_g) \to H_1(S_g)$ defined by
\[ \tau_{\vec v}(\vec w) = \vec w + \hat\imath(\vec w,\vec v) \vec v,\]
where $\hat\imath$ is the algebraic intersection number.   Note that $\tau_{\vec v} = \tau_{-\vec v}$.

The action on $H_1(S_g)$ of a Dehn twist $T_\alpha$ is a transvection $\tau_{[\vec \alpha]}$, where $\vec \alpha$ is $\alpha$ with some choice of orientation and $[\vec \alpha]$ is the associated homology class (see \cite[Proposition 6.3]{primer}); since $\tau_{\vec v} = \tau_{-\vec v}$ this transvection is well defined.  

\begin{lemma}
\label{lemma:kappa}
For $g \geq 2$ we have $\kappa(f_g) = 2g$.  Also, $\kappa(f_2') = 2$.  
\end{lemma}

\begin{proof}

First we treat the case of $f_g$ with $g \notin \{2,4,5,8\}$.  In these cases, $f_g$ has the form:
\[ \left(T_{\alpha_1} \cdots T_{\alpha_n}\right)\left(T_{\beta_1} \cdots T_{\beta_n}\right)^{-1}\]
where $[\vec \alpha_k] = \pm [\vec \beta_k]$ for all $k$ (after choosing arbitrary orientations).  Thus the action of $f_g$ on $H_1(S_g)$ is given by
\[ \left(\tau_{[\vec \alpha_1]} \cdots \tau_{[\vec \alpha_n]}\right)\left(\tau_{[\vec \alpha_1]} \cdots \tau_{[\vec \alpha_n]}\right)^{-1}.\]
Thus the action of $f_g$ on on $H_1(S_g)$ is trivial and so $\kappa(f_g) = 2g$.

The cases of $g \in \{4,5,8\}$ are similar, except in each case there is one additional Dehn twist about a separating simple closed curve.  Since a separating curve is null-homologous, the associated transvection is trivial and it follows again that $\kappa(f_g) = 2g$.

The cases of $f_2$ and $f_2'$ are also similar.  First $f_2$ is a product of Dehn twists about separating curves, so it acts trivially on $H_1(S_2)$.  For the case of $f_2'$ the 2-dimensional subspace of $H_1(S_2)$ that is preserved is the one spanned by the two nonseparating curves in $A_2' \cup B_2'$. 
\end{proof}

\p{Mapping tori and surfaces of genus two in the fibered cone} As mentioned, our examples $f_{g,k}$ will all be derived from the mapping classes $f_g$ already defined.  The main vehicle for doing this is the following theorem of Thurston.

\begin{theorem}[Thurston] \label{T:thurston}
Let $M$ be a hyperbolic 3-manifold of finite volume.  There is a norm $\|\cdot\|$ on $H_2(M)$ with the following properties:
\begin{enumerate}
\item there is a set of maximal open cones $\C_1,\dots,\C_n$ in $H_2(M)$ and a bijection between the set of homotopy classes of connected fibers of fibrations $M \to S^1$ and the set of primitive integral points in the union of the $\C_i$; 
\item the restriction of $\| \cdot \|$ to any $\C_i$ is linear;
\item if $\Sigma$ is a fiber in some fibration $M \to S^1$, then $\|[\Sigma]\| = -\chi(\Sigma)$; and
\item if a surface $\Sigma \subseteq M$ is transverse to the suspension flow associated to some fibration $M \to S^1$, then $[\Sigma]$ lies in the closure $\bar \C_i$ of the corresponding open cone $\C_i$.
\end{enumerate}
\end{theorem}

The norm $\|\cdot\|$ in Theorem~\ref{T:thurston} is now called the Thurston norm and the open cones $\C_i$ are called the \emph{fibered cones} of $M$. 

Our next immediate goal is to use Theorem~\ref{T:thurston} to find surfaces of genus two that represent points in the closure of the fibered cones corresponding to our $f_g$.  We need a preliminary lemma.

\begin{lemma}
\label{lemma:boundary prelim} 
Let $A = \{\alpha_1,\dots,\alpha_m\}$ and $B=\{\beta_1,\dots,\beta_n\}$ be multicurves in $S_g$.  Suppose that $\alpha_1$ and $\beta_1$ are homologous, disjoint simple closed curves that bound a subsurface of genus one, and that $\delta$ is a simple closed curve in $S_g$ with $i(\delta,\alpha_1) = i(\delta,\beta_1) = 1$ and all other intersection numbers $i(\delta,\alpha_k)$ and $i(\delta,\beta_k)$ equal to zero.  Then there are representatives of $\delta$ and $T_AT_B^{-1}(\delta)$ that are disjoint and bound a subsurface of genus one.
\end{lemma}

\begin{proof}
 
We first use the fact that the $T_{\alpha_i}$ pairwise commute and that the $T_{\beta_i}$ pairwise commute to write 
\[ T_AT_B^{-1} = \left(T_{\alpha_2} \cdots T_{\alpha_m}\right)\left(T_{\alpha_1}T_{\beta_1}^{-1}\right)\left(T_{\beta_2}\cdots T_{\beta_n}\right)^{-1}.\]
Since the $\beta_i$ with $i > 1$ are disjoint from $\delta$, it follows that
\[ T_AT_B^{-1}(\delta) = \left(T_{\alpha_2} \cdots T_{\alpha_m}\right)\left(T_{\alpha_1}T_{\beta_1}^{-1}\right)(\delta).\]  
It is straightforward to check under our assumptions that $T_{\alpha_1}T_{\beta_1}^{-1}(\delta)$ and $\delta$ have disjoint representatives that bound a subsurface of genus one.  Applying  $\left(T_{\alpha_2} \cdots T_{\alpha_m}\right)$ to the pair $(T_{\alpha_1}T_{\beta_1}^{-1}(\delta),\delta)$ and using the fact that each $\alpha_i$ with $i > 1$ is disjoint from $\delta$, we conclude that there are representatives of $T_AT_B^{-1}(\delta)$ and $\delta$  that are disjoint and bound a subsurface of genus one, as desired.
\end{proof}

\begin{lemma}
\label{lemma:boundary}
Fix some $g \geq 3$.  Let $M$, $\Sigma$, and $\C$ be the mapping torus, fiber, and fibered cone corresponding to $f_g$.  There are surfaces $\Sigma_0$ and $\Sigma_1$ of genus two in $M$ with $[\Sigma_0]$ and $[\Sigma_1]$ in $\bar \C$ and so that  $[\Sigma_0]$, $[\Sigma_1]$, and $[\Sigma]$ are linearly independent.
\end{lemma}

\begin{proof}

For each $g \geq 3$ we can find a curve $\delta$ in $S_g$ that---together with $A_g$ and $B_g$---satisfies the hypotheses of Lemma~\ref{lemma:boundary prelim}.  In fact, we can find two such curves, $\delta_0$ and $\delta_1$, and we can choose them so that their homology classes are linearly independent:

\bigskip

\labellist
\small\hair 2pt
\pinlabel {$\delta_1$} [ ] at 150 45
\pinlabel {$\delta_0$} [ ] at 80 45
\endlabellist
\centerline{\includegraphics[scale=1]{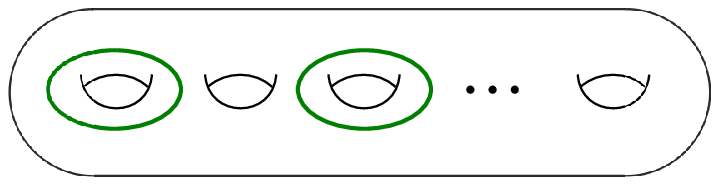}}

\vspace*{2ex}

\noindent(there are many others!).  By Lemma~\ref{lemma:boundary prelim}, $f_g(\delta_i)$ and $\delta_i$ have disjoint representatives that bound a subsurface $\Xi_i$ of genus one for each $i\in\{0,1\}$.  As described in a previous paper by the second and third authors \cite[Proof of Lemma 5.1]{LM}, we can use $\Xi_0$ and $\Xi_1$ to construct surfaces $\Sigma_0$ and $\Sigma_1$ in $M$ that have genus two and are transverse to the suspension flow on $M$ associated to $f_g$; the idea is to drag $\delta_i \subset \Xi_i$ in the direction of the suspension flow of $M$ until it meets the fiber again for the first time, at which point it coincides with $f_g(\delta_i)$ and hence gives the desired closed surface.  Because the $\Sigma_i$ can be made transverse to the suspension flow, the homology classes $[\Sigma_0]$ and $[\Sigma_1]$ both lie in the closure $\bar {\C}$ (see Theorem~\ref{T:thurston}).

Finally, we prove that $[\Sigma_0]$, $[\Sigma_1]$, and $[\Sigma]$ are linearly independent.  Since $\kappa(f_g) = 2g$, the inclusion of $\Sigma \to M$ induces an injection $H_1(\Sigma) \to H_1(M)$.   As the cap product with $[\Sigma]$ is a linear function $H_2(M) \to H_1(M)$ and is given by oriented intersection with $\Sigma$, we have $[\Sigma]  \smallfrown [\Sigma] = 0$, $[\Sigma_0] \smallfrown [\Sigma] = [\delta_0]$, and $[\Sigma_1] \smallfrown [\Sigma] = [\delta_1]$.  It then follows from the fact that $[\delta_0]$ and $[\delta_1]$ are linearly independent that $[\Sigma]$, $[\Sigma_0]$, and $[\Sigma_1]$ are linearly independent.
\end{proof}

Our next goal is to give an analogue of Lemma~\ref{lemma:boundary} for $f_2$ and $f_2'$.  We again require a preliminary lemma.

\begin{lemma}
\label{lemma:prelim}
Let $\delta$ be a nonseparating simple closed curve in $S_2$, and let $\alpha$ be an essential separating simple closed curve in $S_2$ that intersects $\delta$ essentially in two points.  Let $\tilde S_2$ be the infinite cyclic cover of $S_2$ corresponding to $\delta$.  There is a subsurface $\tilde \Sigma_0 \subseteq \tilde S_2$ with genus one and with two boundary components that are connected components of the preimages of $\delta$ and (a representative of) $T_\alpha(\delta)$, respectively.
\end{lemma}

\begin{proof}

Up to homeomorphism, the curves $\delta$ and $\alpha$ are arranged in $S_2$ as in the left-hand side of the following figure:

\vspace*{1ex}

\labellist
\small\hair 2pt
\pinlabel {$\alpha$} [ ] at 37 44
\pinlabel {$\delta$} [ ] at 105 28
\pinlabel {$\tilde \delta_0$} [ ] at 176 60
\pinlabel {$\tilde \delta_1$} [ ] at 213 60
\pinlabel {$\tilde \delta_2$} [ ] at 251 60
\pinlabel {$\tilde \alpha_0$} [ ] at 168 37
\pinlabel {$\tilde \alpha_1$} [ ] at 204 37
\pinlabel {$\tilde \alpha_2$} [ ] at 240 37
\endlabellist
\begin{center}
\noindent \includegraphics[scale=1]{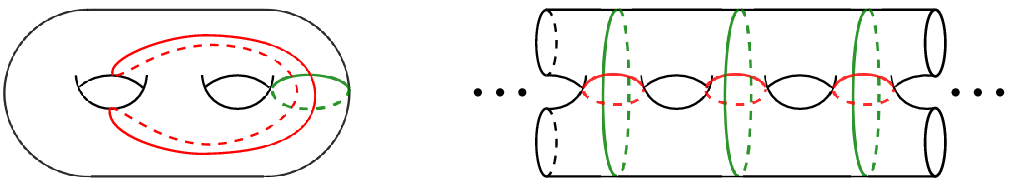}
\end{center}
and their preimages $\tilde \delta$ and $\tilde \alpha$ in $\tilde S_2$ are arranged as in the right-hand side.  Each component of $\tilde \delta$ is separating in $\tilde S_2$ and there is an induced ordering of the components; as such, we label the components as $\{\tilde \delta_i \mid i \in \Z\}$.  Each component of $\tilde \alpha$ intersects exactly one $\tilde \delta_i$ and so the components of $\tilde \alpha$ inherit labels $\tilde \alpha_i$ as in the figure.

Since the covering map $\tilde S_2 \to S_2$ restricts to a homeomorphism on each component of $\tilde \alpha$, the Dehn twist  $T_{\alpha}$ lifts to a multitwist $T_{\tilde \alpha}$ of $\tilde S_2$.  The preimage of $T_{\alpha}(\delta)$ in $\tilde S_2$ is $T_{\tilde \alpha}(\tilde \delta)$.  We claim that $T_{\tilde \alpha}(\tilde \delta_1)$ bounds a subsurface $\tilde \Sigma_0$ of genus one with $\tilde \delta_0$ (assuming $T_\alpha$ was taken to have support on a small regular neighborhood of $\alpha$).  Indeed, we can write the pair $(\tilde \delta_0, T_{\tilde \alpha}(\tilde \delta_1))$ as $(\tilde \delta_0, T_{\tilde \alpha_1}(\tilde \delta_1))$ since $\tilde \alpha_1$ is the only component of $\tilde \alpha$ that intersects $\tilde \delta_1$ and if we apply $T_{\tilde \alpha_1}^{-1}$ to this pair, we obtain the pair $(\tilde \delta_0, \tilde \delta_1)$, which obviously bounds a surface $\tilde \Sigma_0'$ of genus one.  Since $T_{\tilde \alpha_1}$ is a homeomorphism the surface $\tilde \Sigma_0 = T_{\tilde \alpha_1}(\tilde \Sigma_0')$ has genus one.  By construction, the two boundary components of $\tilde \Sigma_0$ are connected components of the preimages of $\delta$ and $T_{\alpha}(\delta)$.
\end{proof}

In Lemma~\ref{lemma:prelim}, it is possible to replace $S_2$ with $S_g$, in which case $\tilde \Sigma_0$ becomes a surface of genus $g-1$ with two boundary components.

\begin{lemma}
\label{lemma:boundary2}
Let $f$ be either $f_2$ or $f_2'$.  Let $M$, $\Sigma$, and $\C$ be the associated mapping torus, fiber, and fibered cone.  
There is a surface $\Sigma_0$ of genus two in $M$ so $[\Sigma_0]$ is linearly independent from $[\Sigma]$.
\end{lemma}

\begin{proof}

For the case of $f_2$ let $\delta$ be the nonseparating simple closed curve in $S_2$ shown in the following figure:

\vspace*{2ex}

\labellist
\small\hair 2pt
\pinlabel {$\delta$} [ ] at 80 42
\endlabellist
\centerline{\includegraphics[scale=1]{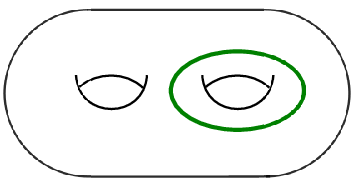}}

\vspace*{2ex}

\noindent and for the case of $f_2'$ let $\delta$ be either of the two nonseparating curves in $B_2'$.  Let $\alpha$ denote either the curve in $A_2$ or $A_2'$, according to our choice of $f$.  Since $\delta$ is disjoint from $B_2$ and $B_2'$ we have $f(\delta) = T_\alpha(\delta)$.  Let $\tilde S_{2}$ be the infinite cyclic cover of $S_{2}$ corresponding to $\delta$ and let $p : \tilde S_{2} \to S_{2}$ be the covering map.   Since $\alpha$ is a separating simple closed curve that intersects $\delta$ essentially in two points, Lemma~\ref{lemma:prelim} provides a subsurface $\tilde \Sigma_0$ of genus one and with two boundary components that are connected components of the preimages of $\delta$ and a representative of $T_\alpha(\delta)$.   We label the preimages of $\alpha$ and $\delta$ in $\tilde S_{2}$ as in Lemma~\ref{lemma:prelim}, and as in the proof of that lemma, we can take $\tilde \Sigma_0$ to lie between $\tilde \delta_0$ and $T_{\tilde \alpha}(\tilde \delta_1)$.  Our goal is to build a map $\psi : \tilde \Sigma_0 \to M$ with image an embedded, smooth closed surface of genus two transverse to the suspension flow. 

We claim there is a smooth function $\psi_1 : \tilde \Sigma_0 \to \R$ with the following properties:
\begin{enumerate}
 \item $\psi_1|\tilde \delta_0=0$ and $\psi_1|T_{\tilde \alpha}(\tilde \delta_1) = 1$,
 \item $\psi_1$ is constant on a neighborhood of $\partial \tilde \Sigma_0$,
 \item if $x,y \in \tilde \Sigma_0$ differ by a nontrivial covering transformation of $\tilde S_{2}$ then $\psi_1(x) \neq \psi_1(y)$.
\end{enumerate}
To prove the claim, we first note that there is a function $\psi_2 : \tilde S_{2} \to \R$ with the following properties:  
\begin{enumerate}
 \item $\psi_2|\tilde \delta_i=i$, 
 \item $\psi_2$ is constant on a neighborhood of $\tilde \delta$,
 \item if $x$ and $y$ in $\tilde S_{2}$ differ by a nontrivial covering transformation, then $|\psi_2(x) - \psi_2(y)| \geq 1$, and 
 \item $\psi_2(\tilde \alpha_1) \subseteq  (1/2,3/2)$.
\end{enumerate}
Such a function can be obtained by equivariantly perturbing the obvious (horizontal) height function on $\tilde S_{2}$, using the picture in the proof of Lemma~\ref{lemma:prelim}.

The function $\psi_1$ is then the restriction to $\tilde \Sigma_0$ of $\psi_2 \circ T_{\tilde \alpha_1}^{-1}$.  The only nontrivial property of $\psi_1$ to check is the third one.  By the fourth property of $\psi_2$ we can choose a closed neighborhood $A$ of $\tilde \alpha_1$ with $\psi_2(A) \subseteq  (1/2,3/2)$, and we can choose $T_{\tilde \alpha_1}^{-1}$ to be supported in this neighborhood.  It follows that $|\psi_1(x) - \psi_2(x)| < 1$ for all $x \in \tilde \Sigma$.  Combining this with the third property of $\psi_2$, we obtain the third desired property of $\psi_1$, finishing the proof of the claim.

Write $M$ as $S_2 \times [0,1] /\!\!\sim$ where $(x,0) \sim (f(x),1)$ and let $\pi : S_{2} \times [0,1] \to M$ denote the projection.  By the claim, we can define $\psi : \tilde \Sigma_0 \to M$ by
\[ \psi(x) = \pi(p(x),\psi_1(x)).\] 
The image is a smooth, embedded surface $\Sigma_0$ of genus two since the two boundary components of $\tilde \Sigma_0$ are identified in the image.  We claim that $\Sigma_0$ is transverse to the suspension flow.  Indeed, the only way this could fail would be if $p$ were not a submersion; but $p$ is a covering map and hence is a submersion.  It follows that $[\Sigma_0]$ lies in $\bar \C$ (Theorem~\ref{T:thurston}(4)).  The linear independence is proven in the same way as in Lemma~\ref{lemma:boundary}.
\end{proof}

\p{Bounds on entropy in a fibered cone} The next lemma will make use of the following theorem of Fried \cite{Fr0,Fr}.

\begin{theorem}[Fried]
\label{theorem:fried}
Let $\C$ be a fibered cone for a fibered 3-manifold.  There is a strictly convex continuous function $\h \colon \C \to \R$ with the following properties:
\begin{enumerate}
\item for all $t > 0$ and $u \in \C$, we have $\h(tu) = \frac 1t \h(u)$;
\item for every primitive integral $u \in \C \cap H_2(M)$, the entropy of the associated monodromy is $\h(u)$; and 
\item as $u \to \partial \C$ within $\C$, we have $\h(u) \to \infty$.
\end{enumerate}
\end{theorem}
See also \cite{Mc} for another proof and for finer properties of the function $\h$.

We will also make use of the normalized entropy function
\[ \bar \h (x) = \|x\| \h(x). \]
By the first property of $\h$ in Theorem~\ref{theorem:fried}, the function $\bar \h$ is constant on rays through 0.  Therefore, as $\h$ is strictly convex, the function $\bar \h$ is convex.

\begin{lemma}
\label{lemma:convex}
Let $\C$ be a fibered cone for a mapping torus $M$.  If $u \in \C$ and $v \in \bar \C$, then $\h(u+v) < \h(u)$.  
\end{lemma}

\begin{proof}

Set $h_0 = \h(u)$.  By properties (1) and (2) of $\h$ in Theorem~\ref{theorem:fried}, the region $\h^{-1}((0,h_0])$ is convex.  The point $u$ is contained in the boundary, the level hypersurface $\h^{-1}(h_0)$.  By properties (1) and (3) of $\h$ in Theorem~\ref{theorem:fried}, $\h^{-1}(h_0)$ is properly embedded and asymptotic to $\partial \C$; see also \cite{Mc}.  

The intersection of $\h^{-1}((0,h_0])$ with the plane $P \subset H^1(M)$ spanned by $u$ and $v$ is a convex region in $P$ bounded by $\h^{-1}(h_0) \cap P$, which we think of as a parameterized curve.  By the previous paragraph, as the parameter tends to $\pm \infty$ the slopes limit to the slopes of the two rays of $\partial \bar {\C} \cap P$:

\bigskip

\centerline{\includegraphics[scale=.5]{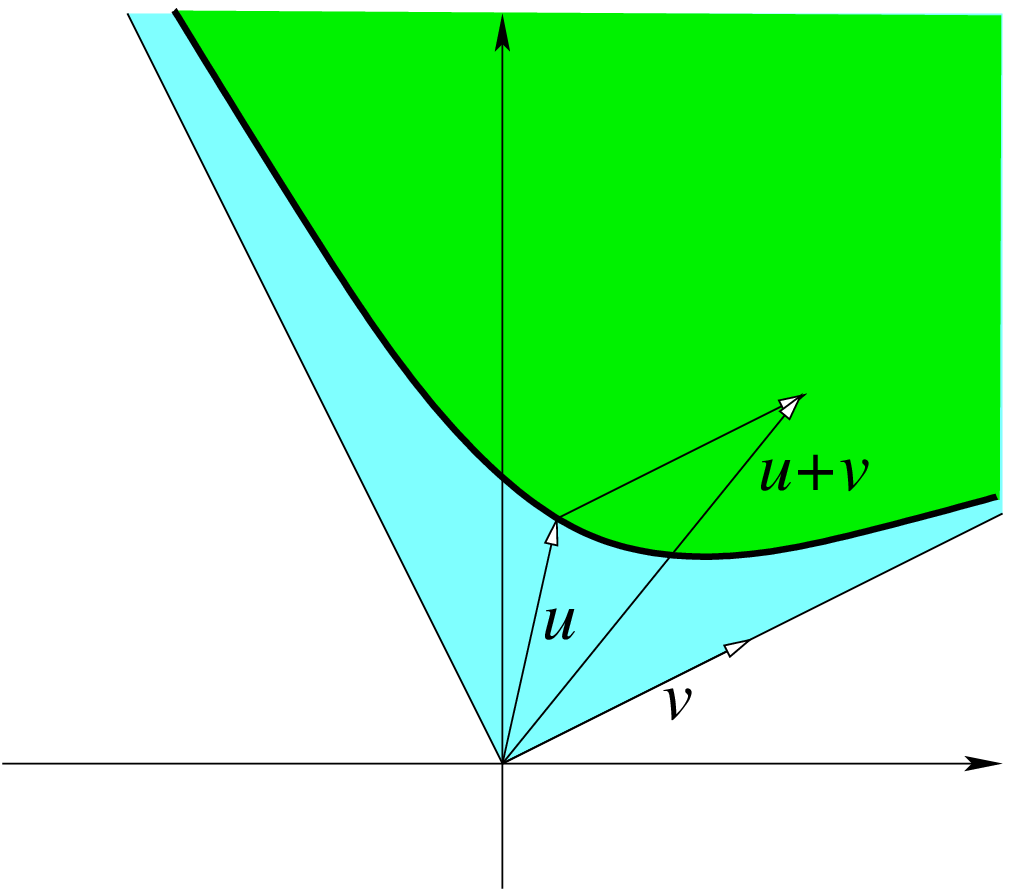}}

\noindent It follows that $u + \bar {\C} \cap P$ (and hence $u+v$) is contained in $\h^{-1}((0,h_0])$.
\end{proof}

\begin{proof}[Proof of Proposition~\ref{P:homology_bound_gen 2}]

For each choice of field $\F$ and each choice of $g$ and $k$ with $g \geq k/2$, we need to construct $f_{g,k} \in \Mod(S_g)$ with $\kappa(f_{g,k};\F) \geq k$ and  
\[h(f_{g,k})  \leq \twoentropy \left(\frac{k + 1}{2g-2}\right ) .\]
Since $\kappa(f) \leq \kappa_\F(f)$ for any field $\F$, it suffices to consider the case $\F = \R$.  

We can dispense with the case $k=0$:  we already stated in the introduction that $L(g,0) < \log(\varphi^4)/(2g-2)$ which is to say that for $g \geq 2$, there is an $f_{g,0} \in \Mod(S_g)$ with $h(f_{g,0}) < \log(\varphi^4)/(2g-2) < \twoentropy/(2g-2)$.  

For $k \geq 2$ even, we will prove something stronger than what is required: we will show that there is an $f_{g,k} \in \Mod(S_g)$ with $\kappa(f_{g,k}) = k$ and
\[h(f_{g,k})  \leq \twoentropy \left(\frac{k}{2g-2}\right ) .\]
The case of $k$ odd follows immediately from this.  Indeed, for $k$ odd, we can then set $f_{g,k} = f_{g,k+1}$.  For then $\kappa(f_{g,k}) = k+1 \geq k$ and $h(f_{g,k}) < \twoentropy\,(k+1)/(2g-2)$. Thus, for the remainder of the proof, we take $k$ to be even.

Let $g_0=k/2$ if $k \geq 4$ and let $g_0=2$ if $k=2$.  Also let $f = f_{g_0}$ if $k \geq 4$ and let $f=f_2'$ if $k=2$.  Let $M$, $\Sigma$, and $\C$ be the associated mapping torus, fiber, and fibered cone.

We can firstly set $f_{g_0,k} = f$, as by Lemma~\ref{lemma:dil} we have
\[ (2g_0-2) h(f) < 2g_0 \cdot \entropy \leq  \twoentropy \cdot k.\]
Assume now that $g > g_0$.  If $g_0 \geq 3$, let $\Sigma_0$ and $\Sigma_1$ be the surfaces in $M$ promised by Lemma~\ref{lemma:boundary}, and if $g_0=2$, let $\Sigma_0$ be the surface promised by Lemma~\ref{lemma:boundary2}; for convenience we set $[\Sigma_1]=0$ in the latter case.  It follows from Lemmas~\ref{lemma:dil} and~\ref{lemma:convex} that $\h([\Sigma])$, $\h([\Sigma] + g_0 \cdot [\Sigma_0])$, and $\h([\Sigma] + [\Sigma_1])$ are all bounded above by $6 \log(2)$.  

We would like to find a primitive integral class $x \in \C$ that has $\|x\| = 2g-2$ and lies in the cone on the convex hull of $[\Sigma]$, $[\Sigma] + g_0 \cdot [\Sigma_0]$, and $[\Sigma] + [\Sigma_1]$.  To do this, we write
\[ 2g-2 = \ell(2g_0-2) + 2r \]
where $\ell > 0$ and $0 < r < g_0$.  If $r$ and $\ell$ are relatively prime we take 
\begin{align*} 
x &= \ell [\Sigma]  + r [\Sigma_0] \\ 
&= \left(\ell - \frac{r}{g_0}\right)[\Sigma] + \left(\frac{r}{g_0}\right)\left( [\Sigma] + g_0 \cdot [\Sigma_0]   \right)
\end{align*}
and if $r$ and $\ell$ are not relatively prime then $r,\ell \geq 2$ and we take
\begin{align*}
x &= \ell [\Sigma]  + (r-1) [\Sigma_0] + [\Sigma_1] \\
& = \left(\ell - 1 - \frac{r-1}{g_0}\right)[\Sigma] + \left(\frac{r-1}{g_0}\right)(    [\Sigma]  + g_0 \cdot [\Sigma_0]) + ([\Sigma] + [\Sigma_1]).
\end{align*}
Note that in the latter case $r \geq 2$ implies $g_0 \geq 3$, and so $[\Sigma_1]\neq 0$.  
By Theorem~\ref{T:thurston}(2) we have $\|x\| = 2g-2$ in either case.  The class $x$ is primitive because $[\Sigma]$ and $[\Sigma_0]$, and $[\Sigma_1]$ are primitive and linearly independent when $g_0 \geq 3$ (Lemma~\ref{lemma:boundary}) and  $[\Sigma]$ and $[\Sigma_0]$ are primitive and linearly independent for $g_0=2$ (Lemma~\ref{lemma:boundary2}).

By Theorem~\ref{T:thurston} the class $x$ is represented by a connected fiber for some fibration of $M$; denote the monodromy by $f_{g,k}$.  Since $\|x\|=2g-2$ the genus of this fiber is $g$, so $f_{g,k} \in \Mod(S_g)$.  Since $b_1(M)=k+1$ (Lemma~\ref{lemma:kappa}) we have $\kappa(f_{g,k}) = k$.  

It remains to bound the entropy $h(f_{g,k})$.  Since $\bar \h$ is convex and constant on rays, $\bar\h(x)$ is bounded above by the values of $\bar\h$ on $[\Sigma]$,  $[\Sigma] + g_0\cdot[\Sigma_0]$, and $[\Sigma] + [\Sigma_1]$.  Applying this fact,  Lemma~\ref{lemma:convex}, and Theorem~\ref{T:thurston}(2), we have:
\begin{align*} 
(2g-2) h(f_{g,k}) = \bar \h(x)  & \leq \max \{ \bar\h([\Sigma]) , \bar\h([\Sigma] + g_0 \cdot [\Sigma_0]),\bar\h([\Sigma] + [\Sigma_1]) \} \\
& \leq \max\{ \| [\Sigma] \|,\| [\Sigma] + g_0 \cdot [\Sigma_0] \|,\| [\Sigma] + [\Sigma_1] \| \} \\&   \qquad \cdot  \max\{ \h([\Sigma]), \h([\Sigma] + g_0 \cdot [\Sigma_0]),\h([\Sigma] + [\Sigma_1]) \} \\
& \leq ((2g_0-2) +2g_0) \cdot h(f) \\
& < 4g_0 \cdot h(f).
\end{align*}
When $k \geq 4$, we have $g_0=k/2$ and $h(f) < \entropy$ (Lemma~\ref{lemma:dil}) and so $h(f_{g,k}) \leq \twoentropy\cdot k/(2g-2)$, as required.  And when $k=2$ we have $g_0=k$ and $h(f) < 2$ (Lemma~\ref{lemma:dil}) and so $h(f_{g,k}) < (4k \cdot 2)/(2g-2) < 12 \log(2) \cdot (3/2g-2)$, as required.
\end{proof}

\section{Counting conjugacy classes}
\label{sec:count}

In this section we prove Theorem~\ref{T:homology_count}, which states that for all $k \geq 0$ there exists constants $c_1,c_2 >0$ so that the number of conjugacy classes of pseudo-Anosov $f \in \Mod(S_g)$ with $\kappa(f) = k$ and $h(f) < \entropy(k+1)/(2g-2)$ is at least $c_1 g^k - c_2$.
The proof is almost identical to the proof of the analogous theorem of the last two authors in the case where there is no restriction on $\kappa(f)$ \cite[Theorem 1.3]{LM}, and so we will refer to the proof in that paper for some of the details.

\begin{proof}[Proof of Theorem \ref{T:homology_count}]

The theorem is vacuously true for $k =0$, so we assume $k  > 0$.  First suppose $k \geq 2$ is even.  If $k \geq 4$, then set $g_0 = k/2$ and let $f = f_{g_0} \colon S_{g_0} \to S_{g_0}$ be the pseudo-Anosov mapping class constructed in Section~\ref{S:examples}.  If $k = 2$, then set $g_0=2$ and let $f = f_2' \colon S_{g_0} \to S_{g_0}$.  According to Lemmas~\ref{lemma:dil} and \ref{lemma:kappa}, $h(f) < \entropy$ and $\kappa(f) = k$.

Let $M$, $\Sigma$, and $\C$ be the mapping torus, fiber, and fibered cone corresponding to $f$, and let $\h \colon \C \to \R$ be the function from Theorem~\ref{theorem:fried}.  By Theorem~\ref{T:thurston}, the restriction of the Thurston norm to $\C$ is the restriction of an integral linear functional $L$ on $H_2(M)$.

Choose any compact neighborhood $K \subset L^{-1}(1) \cap \C$ of $[\Sigma]/\|[\Sigma]\|$ on which the function $\bar\h(x) = \|x\| \h(x)$ is bounded by $6 \log(2) \cdot (2g_0-2) < 6\log(2) \cdot (k+1)$.  Since $\bar\h$ is invariant under scale such a $K$ exists and for the same reason  $\bar\h$ is bounded by $\entropy(k+1)$ on $\R_+ \cdot K$, the cone over $K$.

For every $g \geq g_0$ we set
\[ 
\Omega_{g,k} = \{ [\Sigma] \in (2g-2) \cdot K \mid [\Sigma] \mbox{ is primitive integral}  \}. 
\]
Each element $[\Sigma]$ of $\Omega_{g,k}$ is represented by a fiber of $M$ of genus $g$ for which the monodromy $f_\Sigma \colon \Sigma \to \Sigma$ satisfies $h(f_\Sigma) \leq \entropy \cdot (k+1)/(2g-2)$, as in the statement of the theorem.  We would like to estimate from below the size of $\Omega_{g,k}$.  Since $L$ is integral, we can write $H_2(M;\Z) \cong \Z^k \oplus \Z$, where the first summand is the kernel of the restriction of $L$.  This decomposition extends to a decomposition of $H_2(M)$ as $\R^k \oplus \R$.  By Lemmas~\ref{lemma:boundary} and~\ref{lemma:boundary2} we have $L(H_2(M;\Z)) = 2\Z$.

Let $K_0$ denote the image in $\R^k$ of $K$ under orthogonal projection.  By the previous paragraph, we can identify the integral points of $(2g-2) \cdot K$ with $\Z^k \cap (2g-2) \cdot K_0$ (where the scaling is done in $\R^k$).  As the volume of $(2g-2) \cdot K_0$ grows like $g^k$, it follows that the number of integral points in $(2g-2) \cdot K$ is at least $c_1''g^k - c_2''$ for some $c_1'',c_2'' > 0$ (we must subtract $c_2''$ because, for instance, the set is empty for small values of $g$).   Then by noting that the number of primitive elements make up a definite fraction of the integral elements, we can deduce that for all $g \geq 2$, $|\Omega_{g,k}|$ is bounded below by $c_1'g^k - c_2'$ for some $c_1',c_2' > 0$; see \cite{LM} for the details.

If two points of $\Omega_{g,k}$ have monodromies that are conjugate in $\Mod(S_g)$, then there is a homeomorphism of $M$ taking one fiber to the other.  By Mostow rigidity, such a homeomorphism is homotopic to an isometry of $M$ with respect to its hyperbolic metric.  It follows that the number of conjugacy classes in $\Mod(S_g)$ represented by elements of $\Omega_{g,k}$ is at least $(c_1'g^k - c_2')/N$, where $N$ is the order of the isometry group of $M$.  Setting $c_1 = c_1'/N$ and $c_2 = c_2'/N$ completes the proof in the case of $k$ even.

Now suppose that $k \geq 5$ is odd.  The argument is almost the same as the case when $k$ is even.  Let $g_0 = (k+1)/2$.  Let $f'$ be the mapping class defined in the same way as $f_{g_0}$ except that we leave out one nonseparating curve from the construction, namely, any nonseparating curve disjoint from $\delta_0$.

By calculations similar to those made in Lemmas~\ref{lemma:dil} and~\ref{lemma:kappa} the mapping class $f'$ is pseudo-Anosov with $\kappa(f') = k$ and with 
\[
h(f') \leq h(f) < \entropy (2g_0-2) < \entropy (k+1).
\]
The proof of Lemma~\ref{lemma:boundary} applies and so there is a surface of genus two in the boundary of the fibered cone corresponding to $f'$.  We may apply the argument given above for the case of $k$ even, with $f$ replaced by $f'$.  

\medskip

For the cases of $k=3$ and $k=1$, we need slightly different examples from the ones previously given in $\Mod(S_2)$, namely, the ones given by the product of two positive multitwists about the following pairs of multicurves:

\vspace*{2ex}

\labellist
\small\hair 2pt
\endlabellist
\centerline{\includegraphics[scale=1]{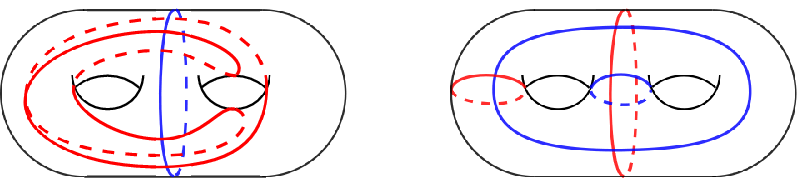}}

\vspace*{2ex}

For the pair of multicurves on the left-hand side we can compute $h(T_AT_B) < \log(34)$ and $\kappa(T_AT_B) = 3$, and for the pair on the right-hand side we can compute $h(T_AT_B) < \log(9)$ and $\kappa(T_AT_B) = 1$.   As the proof of Lemma~\ref{lemma:boundary2} applies in both cases, we can repeat the argument from above to complete the proof.
\end{proof}

As promised in the introduction, we now explain how to promote our Main Theorem to show that the minimal entropy of a pseudo-Anosov element of $\Mod(S_g)$ fixing a subspace of $H_1(S_g;\F)$ of dimension exactly equal to $k$ is comparable to $(k+1)/g$.  For the lower bound there is nothing to do.  For the upper bound and $k>0$ even the examples given in the proof of Proposition~\ref{P:homology_bound_gen 2} have $\kappa_\F = k$, and so there is again nothing to do.   For every odd $k$, we must construct a mapping class as in that proof, except with $\kappa_\F$ equal to $k$ instead of $k+1$.  To do this we simply replace the $f_{g_0}$ in the proof of Proposition~\ref{P:homology_bound_gen 2} with their counterparts described in the proof of Theorem~\ref{T:homology_count}.  The entropies of the latter are bounded above by $\entropy$ since they are obtained by forgetting Dehn twists from the $f_{g_0}$, and they have $\kappa_\F = k$.  

For $k=0$, the techniques of this paper do not apply, since the mapping torus associated to $f$ with $\kappa(f) = 0$ has first betti number equal to 1, and so there are no fibered cones.  However, one can show Penner's examples \cite[p. 448]{Pe} giving the required upper bound have $\kappa_\F=0$ for all $g$ and $\F$.

\appendix

\section{An upper bound for $L(0,g)$} \label{S:limsup to sup}

In this appendix, we prove the following proposition.

\begin{proposition} \label{P:limsup to sup}
For all $g \geq 2$ we have
\[ (2g-2) L(0,g) < \log \left( \varphi^4 \right).\]
\end{proposition}\

The proposition will be proved by exhibiting explicit mapping classes $\psi_g \in \Mod(S_g)$ with $h(\psi_g) < \log \left( \varphi^4\right)/(2g-2)$.  It is known that $L(0,2)$ is the logarithm of the largest root of the polynomial $x^4-x^3-x^2-x+1$, and so $L(0,2) \approx .543533$; see \cite{ChoHam,Zhirov}.  Therefore, it suffices to consider $g \geq 3$.  

Consider the link given by the union of the braid closure of the 3-strand braid $\sigma_1\sigma_2^{-1}$, together with the braid axis and let $M$ be the complement in $S^3$.  Since $\sigma_1\sigma_2^{-1}$ is pseudo-Anosov, $M$ is hyperbolic.  We have the following theorem of Hironaka \cite{Hir}.

\begin{theorem}[Hironaka]
\label{theorem:eko 1}
There are coordinates $(a,b)$ on $H^1(M)$ with the following properties:
\begin{enumerate}
 \item $\|(a,b)\|=2\|(a,b)\|_\infty$
 \item $\C = \{(a,b) \mid b > 0 \mbox{ and } -b < a < b \}$ is a fibered cone
 \item For primitive integral $(a,b) \in \C$, $\h(a,b)$ is the logarithm of the largest root of the polynomial
 \[ L_{a,b}(x) = x^{2b} - x^{b+a} - x^b - x^{b-a} + 1\]
 \item In $\C$, $\bar\h$ attains its minimum on the ray through $(0,1)$, and
 \item for $g \geq 3$ and 
\[ 
\alpha_g = 
\begin{cases}
(1,g+1) & g \equiv 2,5 \mod 6 \\
(3,g+1) & g \equiv 0,1,3,4 \mod 6,  
\end{cases}
\]
the class $\alpha_g \in H^1(M)$ is represented by a surface $S_{g,4}$ of genus $g$ with 4 punctures; the associated monodromy $\psi_g^\circ \colon S_{g,4} \to S_{g,4}$ is pseudo-Anosov and the induced homeomorphism $\psi_g \colon S_g \to S_g$ obtained by filling in punctures is pseudo-Anosov with $h(\psi_g) = h(\psi_g^\circ)$.
\end{enumerate}
\end{theorem}

Proposition \ref{P:limsup to sup} is a consequence of the following proposition.

\begin{proposition} \label{P:eko examples}
For all $g \geq 3$ we have
\[ (2g-2) h(\psi_g) < \log \left( \varphi^4 \right).\]
\end{proposition}

\begin{proof}

We proceed in a series of claims.

\medskip

\emph{Claim 1.} $(2g-2) h(\psi_g) = \left(\frac{g-1}{g+1}\right) \bar\h(\alpha_g).$

\medskip

The claim follows by unraveling definitions and using Theorem~\ref{theorem:eko 1}.

\medskip

\emph{Claim 2.} $(2g-2) h(\psi_g) < \log(\varphi^4)$ for $3 \leq g \leq 5$.

\medskip

The claim is proven by explicit computation using Claim 1 and parts (1) and (3) of Theorem~\ref{theorem:eko 1}.  Specifically, the values of $(2g-2) h(\psi_g)$ for $g=3$, $4$, and $5$ are approximately 1.35, 1.40, and 1.45,
respectively.  As $\log(\varphi^4)$ is approximately 1.92, the claim follows.

\medskip

Before we continue, we define two functions on $(-1,1)$:
\begin{align*}
\h_1(t) & = \bar \h(t,1) \\
F(t) &= (1-\tfrac23t)\h_1(t)
\end{align*}
The function $\h_1$ is the restriction of $\bar\h$ to the points of $\C$ with Thurston norm equal to 2 by part (1) of Theorem~\ref{theorem:eko 1}.  This carries all of the information of $\bar\h$ since the latter is constant on rays.  The relevance of the function $F$ is explained by the following claim.

\medskip

\emph{Claim 3.} $(2g-2) h(\psi_g) \leq F(3/(g+1))$.

\medskip

By Claim 1, the convexity of $\bar \h$, Theorem~\ref{theorem:eko 1}(4), the fact that $\bar\h$ is constant on rays, and the definitions of $\h_1$ and $F$, we have:
\begin{align*}
 (2g-2) h(\psi_g) = \left(\frac{g-1}{g+1}\right) \bar\h(\alpha_g) \leq \left(\frac{g-1}{g+1}\right) \bar\h(3,g+1) 
  = F(3/(g+1))
\end{align*}
as desired.

\medskip

\emph{Claim 4.} $F(t) < \log(\varphi^4)$ for $0 < t \leq 3/7$.

\medskip

First of all, we can easily compute from Theorem~\ref{theorem:eko 1} that $F(0)  = \h_1(0) = \log(\varphi^4)$.  Therefore, it is enough to show that $F'(t)$ is strictly negative for  $0 < t \leq 3/7$.  Applying the product rule we calculate
\[ F'(t)= \h_1'(t) \left(1-\tfrac{2}{3}t\right) -\tfrac{2}{3}\h_1(t). \]
By Theorem~\ref{theorem:eko 1}(4) and the convexity of $\h_1$ we know that $\h_1(t)$ and $\h_1'(t)$ are increasing on $[0,1)$.  Thus, for all $t \in (0,3/7]$ we have
\[ F'(t) \leq \h_1'(3/7)(1-0) - \tfrac23 \h_1(0) = \h_1'(3/7) - \tfrac23 \log(\varphi^4).\]
Because $\h_1'(t)$ is increasing, the mean value theorem implies that $\h_1'(3/7)$ is bounded from above by any difference quotient $\frac{\h_1(3/7+\Delta t)-\h_1(3/7)}{\Delta t}$ with $\Delta t > 0$.  Doing this with $3/7 + \Delta t = 1/2 $ and estimating we obtain
\begin{eqnarray*} F'(t) & \leq & \frac{\h_1(1/2)-\h_1(3/7)}{1/2-3/7} - \frac{2}{3} \log(\varphi^4)\\
& \leq & 1.06-1.28 \\
& < & 0,
\end{eqnarray*}
which gives the claim.

\medskip

Now suppose $g \geq 6$, so that $3/(g+1) \leq 3/7$.  By Claims 3 and 4 we have $(2g-2) h(\psi_g) < \log(\varphi^4)$.  Combined with Claim 2, this completes the proof.
\end{proof}

We remark that by Claim 1 above and Hironaka's result that $\limsup g \cdot h(\psi_g) = \log (  \varphi^2 )$, it follows that $ \log (  \varphi^2 )$ is the least upper bound for the set of normalized entropies of the $\psi_g$.

\bibliographystyle{plain}
\bibliography{geography}

\end{document}